\documentclass[11pt]{article}

\usepackage[top=1in,bottom=1in,left=1.5in,right=1.5in]{geometry}
\usepackage{amsfonts} %\mathbb{R}, etc.
\usepackage{amsmath}
\usepackage{amssymb}
\usepackage{setspace}
\usepackage{subfigure}
\usepackage{url}
\usepackage{booktabs}
\usepackage{ifthen}
\usepackage{tikz}
\usetikzlibrary{calc}
\usepackage{enumerate}

\newcommand{\qed}{\hspace*{\fill} $\Box$}

\newtheorem{thm}{Theorem}[section]

\newtheorem{prop}[thm]{Proposition}

\numberwithin{equation}{section}

\begin{document} 
  
\title{On the number of unlabeled vertices in edge-friendly labelings of graphs}

\author{
Elliot Krop\thanks{Department of Mathematics, Clayton State University, ({\tt ElliotKrop@clayton.edu})}
\and
Sin-Min Lee\thanks{Department of Computer Science, San Jose State University, ({\tt lee.sinmin35@gmail.com})}
\and 
Christopher Raridan\thanks{Department of Mathematics, Clayton State University, ({\tt ChristopherRaridan@clayton.edu}).}}

\date{\today}

\maketitle

\begin {abstract}
  Let $G$ be a graph with vertex set $V(G)$ and edge set $E(G)$, and $f$ be a $0-1$ labeling of $E(G)$ so that the absolute difference in the number of edges labeled $1$ and $0$ is no more than one. Call such a labeling $f$ \emph{edge-friendly}. We say an edge-friendly labeling induces a \emph{partial vertex labeling} if vertices which are incident to more edges labeled $1$ than $0$, are labeled $1$, and vertices which are incident to more edges labeled $0$ than $1$, are labeled $0$. Vertices that are incident to an equal number of edges of both labels we call \emph{unlabeled}.
Call a procedure on a labeled graph a \emph{label switching algorithm} if it consists of pairwise switches of labels. Given an edge-friendly labeling of $K_n$, we show a label switching algorithm producing an edge-friendly relabeling of $K_n$ such that all the vertices are labeled. We call such a labeling \textit{opinionated}.
\\[\baselineskip] 
\\[\baselineskip]
\end {abstract}

\section{Introduction}
 
\subsection{Definitions}
 
For basic graph theoretic notation and definitions see Diestel~\cite{Diest}. All graphs $G(V,E)$ are finite, simple, undirected graphs with vertex set $V$ and edge set $E$. A \textit{labeling} of a graph $G$ with $H \subseteq V \cup E$ is a function $f : H \rightarrow A$ for some set $A$, and if $A = \mathbb{Z}_2 = \{ 0,1 \}$, then the labeling is called \textit{binary}. If $H=E$ ($H=V$), then the labeling $f$ is called an \textit{edge labeling} (\textit{vertex labeling}). Let $i \in \mathbb{Z}_2$. For an edge labeling $f$ of $G$, $f(uv)$ denotes the label on edge $uv$ in $G$. If $f(uv)=i$, we call the edge an $i$\textit{-edge}. The cardinality of $\{ uv \in E : f(uv) = i \}$ is denoted by $e_f(i)$. An edge labeling $f$ is called an \textit{edge-friendly labeling} if $|e_f(0) - e_f(1)| \leq 1$. For any vertex $v$ in $G$, let $N_i(v) = \{ u \in V : f(uv)=i \}$. An edge-friendly labeling $f$ induces a \textit{partial vertex labeling} $f^+ : V \rightarrow \mathbb{Z}_2$ defined by $f^+(v) = 0$ if $|N_0(v)| > |N_1(v)|$; $f^+(v) = 1$ if $|N_0(v)| < |N_1(v)|$; otherwise, $f^+(v)$ is undefined and we say $v$ is \textit{unlabeled}. We say that a vertex $v$ is \textit{trusty} if changing a label on any edge incident with $v$ does not change the induced label on $v$. Let $v_f(i)$ denote the cardinality of $\{ v \in V : f^+(v) = i \}$. A graph $G$ is called an \textit{edge-balanced} graph if there is an edge-friendly labeling $f$ of $G$ satisfying $|v_f(0) - v_f(1)| \leq 1$ and \textit{strongly edge-balanced} if $v_f(0) = v_f(1)$ and $e_f(0) = e_f(1)$. 

A procedure on a labeled graph will be called a \textit{label switching algorithm} if it consists of pairwise switches of labels. Given an edge-friendly labeling of the complete graph on $n$ vertices, $K_n$, we show a label switching algorithm producing an edge-friendly relabeling of $K_n$ such that all the vertices in the induced partial vertex labeling, are labeled.
 
\subsection{History and Motivation}
 
The assignment of binary labels on substructures of graphs is a classical and essential part of the study of graphs (see~\cite{Diest}, for example). In the present context, binary labelings were popularized by Cahit~\cite{Cahit} in the form of cordial labelings as a simplification of graceful~\cite{Rosa} and harmonious labelings~\cite{GrahamSloan}. After years of unsuccessful attempts to prove the existence of graceful and harmonious labelings on trees, Cahit showed that trees are cordial and that complete graphs with at least four vertices are not cordial. Generalizing, one can say that graphs with uniformly few edges (sparse) are more likely to be cordial than graphs with uniformly many edges (dense). Our study begins with a related problem motivated by balanced vertex labelings, introduced by Lee, Liu, and Tan~\cite{LeeLiuTan}, and their extensions to balanced edge labelings by Lee and Kong~\cite{KongLeeEBI}. 

Given a graph $G$ on $n$ vertices and $i \in \{0, \dots, n\}$, does there exist an edge-friendly labeling so that the number of unlabeled vertices is $i$? We present a label switching algorithm that answers the question in the affirmative when $G=K_n$ and $i=0$. When a graph $G$ admits a binary edge labeling so that the number of induced unlabeled vertices is $0$, we say that labeling is \emph{opinionated} and that $G$ is opinionated. Which graphs are opinionated? Cycles, paths, and odd order stars are not opinionated, so unlike the case of cordial labelings, we cannot extend our result to sparse graphs such as trees or $2$-regular graphs; however, we believe that uniformly dense graphs or graphs with high enough connectivity should be opinionated.

\section{Few Labeled Vertices}
 
Chen, Huang, Lee and Liu~\cite{ChenHuangLeeLiu}, produced the following result:

\begin{thm} 
\label{evensize}
If $G$ is a simple connected graph with order $n$ and even size, then there exists an edge-friendly labeling of $G$ so that $G$ is strongly edge-balanced.
\end {thm}

\begin{proof}
Let $G(V,E)$ be a simple connected graph of even size. By the Handshaking Lemma, $G$ has an even number of vertices with odd degree, say $u_1$, $u_2$, \ldots, $u_{2k-1}$, $u_{2k}$, for some integer $k \geq 0$. Suppose we add $k$ new vertices to $G$--call them $v_1^*$, $v_2^*$, \ldots, $v_k^*$--and form a new graph $G^*(V^*,E^*)$ as follows: $V^* = V \cup \{ v_1^*, v_2^*, \ldots, v_k^* \}$ and $E^* = E \cup \{ u_1v_1^*, v_1^*u_2, u_3v_2^*, v_2^*u_4, \ldots, u_{2k-1}v_k^*, v_k^*u_{2k} \}$. 

Since all vertices of $G^*$ have even degree, $G^*$ is an Eulerian graph with even size. Let $C$ be an Eulerian tour of $G^*$ and label the edges of $C$ from $\mathbb{Z}_2$ alternately. In $G^*$, delete the vertices $v_1^*$, $v_2^*$, \ldots, $v_k^*$ and the edges $u_1v_1^*$, $v_1^*u_2$, $u_3v_2^*$, $v_2^*u_4$, \ldots, $u_{2k-1}v_k^*$, $v_k^*u_{2k}$. The remaining graph has an underlying structure which is isomorphic to $G$, but with a strongly edge-balanced labeling. \qed
\end{proof} 
\\[\baselineskip]
Notice that in the labeled copy of $G$ from the above proof, all the vertices of even degree are unlabeled since every time an edge was traversed along $C$ ``into'' such a vertex, an edge was traversed along $C$ ``out'' of that vertex, and the two edges had different labels by construction. Moreover, by the Handshaking Lemma, the number of vertices of odd degree is always even. Hence, the above proof implies the following result:
 
\begin{prop}
Every finite simple graph of even size admits an edge-friendly labeling which is also edge-balanced, so that all vertices but those of odd degree remain unlabeled.
\end{prop}

\section{All Labeled Vertices}
 
\begin{thm} 
\label{all-labeled}
For odd integers $n \geq 7$, there exists an edge-friendly labeling of $K_n$ such that all the vertices are labeled; that is, $K_n$ is opinionated.
\end{thm} 
 
\begin{proof}
Let $G=K_n$ as in the statement of the theorem with an edge-friendly labeling $f$. Let $P$ be the set of unlabeled vertices, say $|P| = p$, and consider the induced complete subgraph $H = K_p$ on $P$. We provide a label switching algorithm that produces an edge-friendly relabeling of $G$ starting from $f$ such that all vertices are labeled. At each step of the algorithm the newly obtained labeling (that is, the relabeling) is still called $f$. The number of unlabeled vertices decreases and we remove any newly labeled vertices from $P$, and we will continue to call this new, smaller set of unlabeled vertices $P$. When $p=0$, there are no unlabeled vertices and the algorithm terminates, so assume $p > 0$.
 
\textbf{Step 1}: If $H$ does not contain a pair of independent edges with different labels, then go to Step 2. Otherwise, choose such a pair of edges in $H$ and switch the labels on these edges. The vertices incident with these edges are now labeled and the labels of the other vertices in $G$ have not changed. This reduces the order of $H$ by four and we repeat Step 1.
 
\textbf{Step 2}: If $H$ contains no pair of incident edges with different labels, then go to Step 3. Otherwise, choose a pair of incident edges in $H$ that have different labels. Such a pair of edges must form a $3$-path, say $xyz$. Switch the labels on $xy$ and $yz$. Vertices $x$ and $z$ are now labeled, but $y$ remains unlabeled, and the labels on the other vertices in $G$ have not changed, thus reducing the order of $H$ by two. Repeat Step 2.
 
\textbf{Step 3}: If $p=1$, then go to Step 4. Otherwise, $p \not= 0,1$, and all the edges of $H$ share the same label. Choose an edge $xy$ in $H$, and without loss of generality, suppose $f(xy)=1$. Since $x$ and $y$ are unlabeled, we can find a vertex $z$ in $G$ adjacent to $x$ so that $f(xz)=0$ and a vertex $w$ in $G$ adjacent to $y$ so that $f(yw)=0$. Note that since $n \geq 7$ we can choose $z$ and $w$ distinct.
 
Suppose $z$ or $w$ is trusty (unlabeled). Without loss of generality, assume $z$ is trusty (unlabeled) and switch the labels on edges $xz$ and $xy$. The labels on $x$ and $z$ do not change ($x$ does not change), but $f^+(y)=0$, reducing the order of $H$ by one (two). Repeat Step 3.   
  
Suppose that $z$ and $w$ are untrusty.

\textit{Case 1.} If $|N_0(y) \cap N_1(z)| \geq 1$, say $u \in N_0(y) \cap N_1(z)$, then the edges of the $4$-cycle $xyuzx$ are alternately labeled. Switching the labels of these edges does not change the vertex labels. However, under the new labeling, $f(xy)=0$. If $xy$ is not the only edge in $H$, we return to Step 1; otherwise, return the $4$-cycle to its original labeling.
 
Suppose $H$ contains only one edge: the edge $xy$. If $f(yz)=1$, then switch the labels on $xz$ and $yz$. Notice that $f^+(x)=1$ and $f^+(y)=0$ and we have no unlabeled vertices. If $f(yz)=0$, switch the labels on the edges of the $4$-cycle $xyuzx$ and then switch the labels on $xz$ and $yz$. This time $f^+(x)=0$ and $f^+(y)=1$ and we have no unlabeled vertices.
  
\textit{Case 2.} If $|N_0(y) \cap N_1(z)| = 0$, we choose a vertex $v \in N_1(z)$. Notice that $f(zw)=0$ and $f(yv)=1$, and recall that $w$ is untrusty. If $v$ is trusty, switch the labels on $zv$ and $xz$. The labels on $v$ and $z$ do not change but $f^+(x)=1$, and the number of unlabeled vertices decreases by $1$. If $v$ is untrusty, we switch the labels on the edges of the 4-cycle $ywzvy$ and notice that the switch does not change the labels on $y$ or $z$.
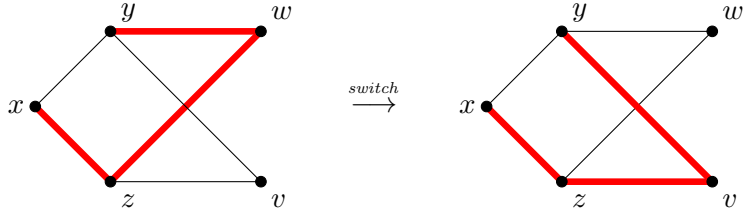
\begin{center}
\begin{tikzpicture}
\draw[line width=2.5pt, color=red] (0,1) -- (1,0);
\draw (0,1) -- (1,2);
\draw (1,0) -- (3,0);
\draw[line width=2.5pt, color=red] (1,2) -- (3,2);
\draw[line width=2.5pt, color=red] (1,0) -- (3,2);
\draw (1,2) -- (3,0);
\coordinate[label=right:$x$] (1) at (-.5,1);
\coordinate[label=right:$z$] (2) at (1,-.25);
\coordinate[label=right:$y$] (3) at (1,2.25);
\coordinate[label=right:$w$] (4) at (3,2.25);
\coordinate[label=right:$v$] (5) at (3,-.25);

\filldraw[black]
	(0,1) circle (2pt)
	(1,0) circle (2pt)
	(1,2) circle (2pt)
	(3,0) circle (2pt)
	(3,2) circle (2pt);
 
\node at (4.5,1) {$\longrightarrow$};
\node at (4.5,1.25) {\tiny\emph{switch}};

\draw[line width=2.5pt, color=red] (6,1) -- (7,0);
\draw (6,1) -- (7,2);
\draw[line width=2.5pt, color=red] (7,0) -- (9,0);
\draw (7,2) -- (9,2);
\draw (7,0) -- (9,2);
\draw[line width=2.5pt, color=red] (7,2) -- (9,0);
\coordinate[label=right:$x$] (1) at (5.5,1);
\coordinate[label=right:$z$] (2) at (7,-.25);
\coordinate[label=right:$y$] (3) at (7,2.25);
\coordinate[label=right:$w$] (4) at (9,2.25);
\coordinate[label=right:$v$] (5) at (9,-.25);

\filldraw[black]
	(6,1) circle (2pt)
	(7,0) circle (2pt)
	(7,2) circle (2pt)
	(9,0) circle (2pt)
	(9,2) circle (2pt);
\end{tikzpicture}
\end{center}  
   
\begin{itemize}
	\item If the labels on $v$ and $w$ do not change after the switch, then $v$ and $w$ both become trusty, and we can switch the labels on $xy$ and $yv$. Therefore, $f^+(x)=0$ and $y$ is unlabeled. If $xy$ is the last edge in $H$, then we continue to Step 4; otherwise, the edges of $H$ do not all share the same labels under this labeling and we return to Step 1.
	
	\item Suppose without loss of generality that the label on $w$ does not change after the switch, but the label on $v$ changes. This means that $f^+(w)=1$. Switch the edge labels on the $4$-cycle $ywzvy$ back to their original state and then switch the labels on $xy$ and $yw$, making $f^+(x)=0$ and leaving the other vertex labels unchanged. Repeat Step 3.
  
  \item Suppose that the labels on $w$ and $v$ change after the switch. Since the degree of $w$ is even, $\Big||N_0(w)| - |N_1(w)|\Big|=2$, so $w$ and $v$ remain untrusty. We have switched the labels of two edges incident with $w$ from $0$ to $1$ and two edges incident with $v$ from $1$ to $0$, so we know that after the switch, $f^+(w)=1$ and $f^+(v)=0$. Switch the labels on the edges of the $4$-cycle $ywzvy$ back to the original labeling and notice that $f^+(w)=0$ and $f^+(v)=1$. If $f(vw)=1$, then switching the labels on $yw$, $wv$, $vz$, and $zx$ does not change the labels on $w$ or $z$ but does change the labels on $v$, $x$, and $y$ so that $f^+(v)=0$ and $f^+(x)=f^+(y)=1$. If $f(vw)=0$, then switch the labels on $yw$, $wv$, $wz$, $xy$, $yv$, and $vz$ and notice that under this new labeling, $f^+(w)=1$, $f^+(x)=f^+(y)=0$, $f^+(z)$ remains unchanged, and $v$ becomes unlabeled. Hence, the set of unlabeled vertices is reduced by one, and we return to Step 1.
\end{itemize}
  
\textbf{Step 4}: Since $p=1$, let $v$ be the unique vertex of $H$. Choose vertices $u_1, u_2 \in N_0(v)$. If $f(u_1u_2)=1$, and $u_1$ is trusty or $f^+(u_1)=0$, switch the labels on $vu_2$ and $u_1u_2$. Then the labels on $u_1$ and $u_2$ do not change, but $v$ is now labeled $1$. By analogy we can consider the case when $v_1, v_2 \in N_1(v)$.
  
Consider the case when $u_1, u_2 \in N_0(v)$, $v_1, v_2 \in N_1(v)$, $f(u_1u_2)=1$, $f(v_1v_2)=0$, $u_1$, $u_2$, $v_1$, $v_2$ are untrusty, $f^+(u_1)=f^+(u_2)=1$, and $f^+(v_1)=f^+(v_2)=0$. We switch the labels on $vu_1$ and $vv_1$ making $u_1$ and $v_1$ trusty. By switching the labels on $vu_2$ and $u_1u_2$, we obtain $f^+(v)=1$ and none of the labels on all other vertices in $G$ change. 
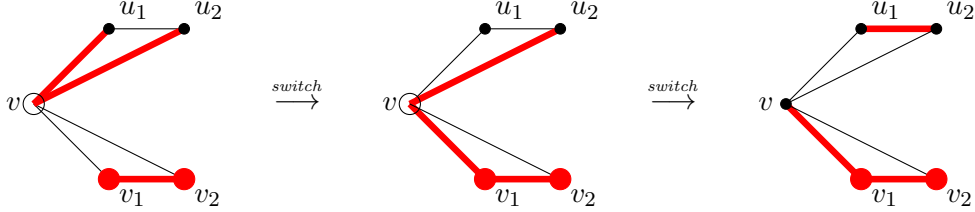
\begin{center}
\begin{tikzpicture}
\draw (0,1) -- (1,0);
\draw[line width=2.5pt, color=red] (0,1) -- (1,2);
\draw (0,1) -- (2,0);
\draw (1,2) -- (2,2);
\draw[line width=2.5pt, color=red] (0,1) -- (2,2);
\draw[line width=2.5pt, color=red] (1,0) -- (2,0);
\coordinate[label=right:$v$] (1) at (-.5,1);
\coordinate[label=right:$v_1$] (2) at (1,-.25);
\coordinate[label=right:$u_1$] (3) at (1,2.25);
\coordinate[label=right:$u_2$] (4) at (2,2.25);
\coordinate[label=right:$v_2$] (5) at (2,-.25);

\draw[black]
 	(0,1) circle (4pt);
 
\filldraw[black]
	(1,2) circle (2pt)
	(2,2) circle (2pt);

\filldraw[red]
	(1,0) circle (4pt)
	(2,0) circle (4pt);

\node at (3.5,1) {$\longrightarrow$};
\node at (3.5,1.25) {\tiny\emph{switch}};
 
\draw [line width=2.5pt, color=red] (5,1) -- (6,0);
\draw (5,1) -- (6,2);
\draw[line width=2.5pt, color=red] (6,0) -- (7,0);
\draw (6,2) -- (7,2);
\draw[line width=2.5pt, color=red] (5,1) -- (7,2);
\draw (5,1) -- (7,0);
\coordinate[label=right:$v$] (1) at (4.5,1);
\coordinate[label=right:$v_1$] (2) at (6,-.25);
\coordinate[label=right:$u_1$] (3) at (6,2.25);
\coordinate[label=right:$u_2$] (4) at (7,2.25);
\coordinate[label=right:$v_2$] (5) at (7,-.25);

\draw[black]
 	(5,1) circle (4pt);

\filldraw[black]
	(6,2) circle (2pt)
	(7,2) circle (2pt);

\filldraw[red]
	(6,0) circle (4pt)
	(7,0) circle (4pt);

\node at (8.5,1) {$\longrightarrow$};
\node at (8.5,1.25) {\tiny\emph{switch}};

\draw[line width=2.5pt, color=red] (10,1) -- (11,0);
\draw (10,1) -- (11,2);
\draw[line width=2.5pt, color=red] (11,0) -- (12,0);
\draw[line width=2.5pt, color=red] (11,2) -- (12,2);
\draw (10,1) -- (12,2);
\draw (10,1) -- (12,0);
\coordinate[label=right:$v$] (1) at (9.5,1);
\coordinate[label=right:$v_1$] (2) at (11,-.25);
\coordinate[label=right:$u_1$] (3) at (11,2.25);
\coordinate[label=right:$u_2$] (4) at (12,2.25);
\coordinate[label=right:$v_2$] (5) at (12,-.25);

\filldraw[black]
 	(10,1) circle (2pt);

\filldraw[black]
	(11,2) circle (2pt)
	(12,2) circle (2pt);

\filldraw[red]
	(11,0) circle (4pt)
	(12,0) circle (4pt);
\end{tikzpicture}
\end{center}

The following argument shows that if such $u_1$, $u_2$, $v_1$, and $v_2$ with $f(u_1u_2)=1$ and $f(v_1v_2)=0$ do not exist, then we may switch the labels on the appropriate edges so that the labeling remains edge-friendly and $v$ becomes labeled. Let $A_i$ be the induced subgraph on $N_i(v)$ with the induced labeling from $G$. We first bound the number of $0$-edges from $A_0$ to $A_1$; in particular, we show that if there are many such edges, then some vertex in $A_0$ must be trusty, and we can perform switches to label $v$. Suppose $u_1 \in A_0$ and $v_1,v_2 \in A_1$ and $f(u_1v_1) = f(u_1v_2) = 0$. If $A_0$ contains a $1$-edge or $A_1$ contains a $0$-edge, then by the above argument, we can make the necessary switches to label $v$ since we only have the cases from the beginning of Step 4. Hence, assume to the contrary that every edge in $A_0$ is a $0$-edge, which implies $|N_0(u_1)| \geq \frac{n-1}{2}+2$; in other words, $u_1$ is trusty. Switch the labels on $u_1v_1$ and $vv_1$. Then the labels on $u_1$ and $v_1$ do not change, but $f^+(v)=0$. Thus, we need only consider the case when the maximum number of $0$-edges from any vertex in $A_0$ to $A_1$ is $1$. The same argument holds for $1$-edges from $A_1$ to $A_0$. Hence, for all graphs with an edge-friendly labeling and more than $2|A_0|$ edges between $A_0$ and $A_1$ we can find $u_1$, $u_2$, $v_1$, and $v_2$ so that $u_1, u_2 \in N_0(v)$, $v_1, v_2 \in N_1(v)$, $f(u_1u_2)=1$, and $f(v_1v_2)=0$, or $u_1$ is trusty and we can switch edges appropriately to label $v$. Since the number of edges between $A_0$ and $A_1$ is $|A_0|^2 > 2|A_0|$ for $|A_0| > 2$, we have shown the result for $n\geq 7$. \qed
\end{proof}
\\[\baselineskip]  
To see that the cases $n=3$ and $n=5$ must be excluded from Theorem~\ref{all-labeled}, notice that for $A_0$ (as defined in the above proof), $|A_0|^2 > 2|A_0|$ except when $|A_0| \leq 2$, which can occur only if $G=K_3$ or $G=K_5$. 

Any edge-friendly labeling of $K_3$ induces a partial vertex labeling in which two vertices always remain unlabeled. Let $v$ be a vertex in $K_5$ with an edge-friendly labeling $f$ and $i\in \mathbb{Z}_2$. If $|N_i(v)|=4$, then we delete $v$ and its incident edges from $K_5$. In the remaining $K_4$, we have exactly one edge $e$ with $f(e)=i$, and the vertices incident with $e$ in $K_5$ are unlabeled. If $|N_i(v)|=2$, then $v$ is unlabeled. Therefore assume there is no vertex $v$ in $K_5$ such that $|N_i(v)|=4$ or $|N_i(v)|=2$; that is, assume $|N_i(v)|=1$ or $|N_i(v)|=3$ for every vertex $v$ in $K_5$. However, every vertex in $K_5$ has degree $4$ and $f$ is an edge-friendly labeling, which contradicts the assumption.    
  
Let $v$ be any unlabeled vertex in a graph $G$ and for $i \in \mathbb{Z}_2$ let $A_i$ be defined as in the above proof. Then the following calculation may be effective for lower density graphs:  
\begin{align}
\nonumber 
\left| 0\mbox{-edges in } A_1 \right| 
&= \left| 0\mbox{-edges in } G \right| - \left| 0\mbox{-edges in } A_0 \right| \\
&\quad - \left| 0\mbox{-edges incident to } v \right| - \left| 0\mbox{-edges between } A_0 \mbox{ and } A_1 \right| \label{eqn:eqn1}
\end{align}
Applying~(\ref{eqn:eqn1}) to $G=K_n$ with $n\geq 7$, we obtain
\begin{enumerate}
	\item For $n \geq 9$ and $n \equiv 1 \pmod 4$,
		\begin{align}
  		\nonumber 
  		\left| 0\mbox{-edges in } A_1 \right| 
  		&\geq \frac{{n\choose 2}}{2} - {{\frac{n-1}{2}}\choose{2}} - \frac{n-1}{2} - {\frac{n-1}{2}} \\ 
  		&\geq 1.
  			\label{eqn:eqn2}
  	\end{align}
  
  \item For $n \geq 11$ and $n \equiv 3 \pmod 4$,
		\begin{align}
  		\nonumber 
  		\left| 0\mbox{-edges in } A_1 \right| 
  		&\geq \frac{{n\choose 2}-1}{2} - {{\frac{n-1}{2}}\choose{2}} - \frac{n-1}{2} - {\frac{n-1}{2}} \\ 
  		&\geq 1.
  			\label{eqn:eqn3}
  	\end{align}	
\end{enumerate}
The same calculation holds for the $1$-edges in $A_0$. Hence we can always find the desired $u_1$, $u_2$, $v_1$, and $v_2$ required in Step 4.
  
Notice that the above argument could be applied without Steps 1 through 3 for graphs of large enough order and high enough density, where the degree of the unlabeled vertices are large. This reasoning may lead to further asymptotic results.
 
We state some general questions: 
\begin{enumerate}
	\item Which graphs are opinionated?
	
	\item Classify constants $c$ and graphs $G$ which admit edge-friendly labelings so that the induced vertex labeling produces no more than $c$ unlabeled vertices.
\end{enumerate}

\section{Acknowledgements}

We would like to thank the anonymous referees for their valuable and insightful comments which greatly improved the arguments and presentation of this paper.
 
\bibliographystyle{plain}

\end{document}